\documentclass[reqno]{amsart}
\usepackage{amssymb}
\usepackage{mathrsfs}
\usepackage{amsmath,amstext,amsfonts,verbatim}

\usepackage[latin1]{inputenc}

\usepackage{amsmath,amsthm,amssymb}

\usepackage{amsfonts}

\usepackage{amsmath,amssymb}
\usepackage{graphicx}
\usepackage{color}
\usepackage{indentfirst}

\newtheorem{theorem}{Theorem}[section]

\theoremstyle{definition}
\newtheorem{definition}[theorem]{Definition}

\newtheorem{prop}{Proposition}[section]

\theoremstyle{remark}

\numberwithin{equation}{section}

\newcommand{\neweq}[1]{\begin{equation}\label{#1}}

\def\phi{\varphi}

\def\incep{\left\{\begin{array}{cl} }
 \def\termin{\end{array}\right. }
\def\2af{2^*_\alpha}

\begin{document}

\title [Conditional entropy for amenable group actions]{\textbf{Conditional entropy for amenable group actions}}

\author{Yuan Lian $^{\ast}$}
\address{College \  of \ Mathematics \  and \ Statistics\ , Taiyuan Normal University, Taiyuan, 030619, China}
\email{andrea@tynu.edu.cn}
\thanks{$^{\ast}$Corresponding author: andrea@tynu.edu.cn}

%\author{Xiaojun Huang}
%\address{College \  of \ Mathematics \  and \ Statistics\ ,  Chongqing\  University, Chongqing,  401331, China}
%\email{hxj@cqu.edu.cn}

\author{Bin Zhu}
\address{School \ of \ Mathematics \ and \ Statistics\ , Chongqing Technology and Business University, Chongqing, 400067, China}
\email{binzhucqu@163.com}
\thanks{The research was supported by Fundamental Research Program of Shanxi Province (No.20210302123322), the Science and Technology Research Program of Chongqing Municipal Education
Commission (Grant No.KJQN202400826) and Research Project of Chongqing Technology and Business University(No. 2256015).}
%\thanks{$^*$ Corresponding author:andrea@tynu.edu.cn.}


% General info
%\subjclass[2010]{37A35, 37B40}

\date{}
%    \subjclass is required.

\begin{abstract}
Let $G$ be an infinite discrete countable amenable  group acting continuously on a Lebesgue space $(X,\mathcal{B},\mu)$. In this article, using partition and factor-space, the conditional entropy of the action $G\curvearrowright^{T} (X, \mu)$ is defined. We introduction some properties of conditional entropy for amenable group actions and the corresponding decomposition theorem is obtained.
\end{abstract}

\maketitle


%------------------------------------------------------------------------------
\section{Factor space}

For the purpose of this article, this part mainly reviews some basic knowledge about factor space, which can be found in \cite{Roh,Roh1}.

Let $(X,\mathcal{B},\mu)$ be a Lebesgue space. If the elements of a collection are disjointed and combined to form $X$, we call it a partition of $X$. If a subset of $X$ can be represented as the union of some elements in a partition $\alpha$ of $X$, then we call it the $\alpha$-set of $X$. We can define an equivalence relation on $\mathcal{B}$ as follows: $A$ and $B$ are equivalent if and only if $\mu(A\triangle B)=0$. This article uses $\overline{\mathcal{B}}$ to represent this set of equivalence classes. The operations in sets in $\mathcal{B}$ of countable union, countable intersection, and subtraction can be passed to the same operation on classes $\overline{\mathcal{B}}$  so that  $\overline{\mathcal{B}}$  is a subalgebra. If the subset of $\overline{\mathcal{B}}$ is closed for the above operation, it is called a sub-$\sigma$-algebra of $\overline{\mathcal{B}}$.

For any measurable partition $\alpha$ , use $\overline{\mathcal{B}}(\alpha)$  to represent a sub-$\sigma$-algebra of $\overline{\mathcal{B}}$ consisting of measurable $\alpha$-sets classes. If $\overline{\mathcal{B}}(\alpha)=\overline{\mathcal{B}}(\alpha')$, then $\alpha=\alpha'$, and for any sub-$\sigma$-algebra of $\overline{\mathcal{B}}$, there exists a measurable partition $\alpha$ such that $\overline{\mathcal{B}}(\alpha)$  equals this subalgebra. Therefore, there is a one-to-one correspondence between the sub-$\sigma$-algebras of $\overline{\mathcal{B}}$ and the classes of mod 0-equal measurable partitions.

For any point $x$ in $X$, we will use $\alpha(x)$ to represent the elements in measurable partition $\alpha$ that contain $x$. Let $\alpha,\beta$ be two measurable partitions of $X$. We use $\alpha\leq\beta$ to represent $\beta(x)\subset\alpha(x)$ for $\mu$-almost every $x\in X$.

The factor-space of a compact metric space $X $ with respect to a partition $\alpha$  is a quotient space, constructed by collapsing each subset in the partition into a single point and the corresponding measures are defined as follows: let $p$ be a mapping that maps each point $x$ in $X$ to an element in the partition $\alpha$ that includes this point; a set $A$ is measurable if $p^{-1}(A)$ is measurable in $X$, and we define the measure $\mu_{\alpha}(A)=\mu(p^{-1}(A))$ of $A$. We denote this factor-space by $X/\alpha$. When $\alpha$ is a measurable partition, $X/\alpha$ is a Lebesgue space. This construction depends on the properties of $\alpha$ to ensure the resulting space retains desirable topological structure. The factor-space $X/\alpha$ is the set of equivalence classes under the relation $x\sim y \iff x, y \in P_i$ for some $P_i \in \alpha$. Points in  $X/\alpha$ correspond to the partition elements $P_i$.

%------------------------------------------------------------------------------
\section{Group actions}

The following definitions can be referred to \cite{KL1}. In recent decades, the entropy theory for an action of an amenable group had obtained an extensive development \cite{KL,KL1,KY,OJ,OW}.

An amenable group $G$ (e.g. $\mathbb{Z}^{d}$, abelian groups, or solvable groups) is a countable discrete group that admits a F{\o}lner sequence ( a sequence of finite subsets $F_{n} \subset G$ that $approximate$ invariance under translation). Formally, for all  $g \in G$,
$$
\lim_{n \to \infty} \frac{|g F_n \triangle F_n|}{|F_n|} = 0,
$$
where \( \triangle \) denotes symmetric difference. This property allows averaging over the group in a consistent way. Throughout this paper, $G$ is an infinite countable discrete amenable group. Its identity element will always be denoted by $e$.

For a probability-measure-preserving (p.m.p.) action of $G$ on a space $ (X, \mu)$, the Kolmogorov-Sinai entropy generalizes to mean entropy (or entropy rate) using F{\o}lner sequences $\{F_{n}\}$. For a partition $\xi$ of $X$, the entropy is:
$$
h_\mu(G, \xi) = \lim_{n \to \infty} \frac{1}{|F_n|} H_\mu\left( \xi^{F_{n}}\right),
$$
where $\xi^{F_{n}}=\bigvee_{g \in F_n} g^{-1} \xi $ and $ \bigvee$ denotes the join of partitions. This measures the asymptotic uncertainty per group element.

\begin{definition}
By an action of the group $G$ on $X$ we mean a map $T:G\times X\longrightarrow X$ such that, writing the first argument as a subscript, $T_{s}(T_{t}(x))=T_{st}(x)$ and $T_{e}(x)=x$ for all $x\in X$ and $s,t\in G$. Most of the time we will write the image of a pair $(s,x)$ written as $sx$.
\end{definition}

\begin{definition}
By a p.m.p.(probability-measure-preserving)action of $G$, we mean an action of $G$ on a standard probability space $(X, \mu)$ by measure-preserving transformations. In this case, we will combine together the notion and simply write $G\curvearrowright^{T} (X, \mu)$.
\end{definition}

%------------------------------------------------------------------------------
\section{Mean conditional entropy}

The content regarding mean conditional entropy used in this article is as follows, whose proof can be found in \cite{Roh,LQ}.

There is a very important property about measurable partitions in a Lebesgue space $(X,\mathcal{B},\mu)$ is that each partition $\alpha$ has a unique measurement system $\{\mu_{A}\}_{A\in \alpha}$ that satisfies the following two conditions:

\begin{enumerate}
\item $(A,\mathcal{B}|_{A},\mu_{A})$ be a Lebesgue space for $\mu_{\alpha}-$a.e. $A\in X/\alpha$;

\item for any $C\in \mathcal{B}, \mu_{A}(C\cap A)$ is measurable on $X/\alpha$ and

$$\mu(C)=\int_{X/\alpha}\mu_{A}(C\cap A)d\mu_{\alpha}.$$

\end{enumerate}

Such a system of measures $\{\mu_{A}\}_{A\in \alpha}$ is called a canonical system of conditional measures of $\mu$ associated with $\alpha$. The uniqueness mentioned above implies that $\mu_{A}=\mu'_{A}$ for $\mu_{\alpha}-$a.e. $A\in X/\alpha$ between any two systems $\{\mu_{A}\}_{A\in \alpha}$ and $\{\mu'_{A}\}_{A\in \alpha}$ that satisfy the two conditions listed above.

Let $\{\mu_{A}\}_{A\in \alpha}$ be a canonical system of conditional measures of $\mu$ associated with measurable partition $\alpha$. According to the knowledge of measure theory, for any $f\in L^{1}(X,\mathcal{B},\mu)$, the section $f_{A}$ defined as
$$
f_{A}(x)=f(x), \text{if} \  x\in A
$$
is integrable on  $(A,\mathcal{B}|_{A},\mu_{A})$ for $\mu_{\alpha}-$a.e. $A\in X/\alpha, \int_{A}f_{A}d\mu_{A}$ is measurable on $X/\alpha$ and
$$
\int_{X}f(x)d\mu=\int_{X/\alpha}(\int_{A}f_{A}d\mu_{A})d\mu_{\alpha}.
$$

\begin{definition}
Let $\alpha$ be a measurable partition of $X$ and let $A_{1},A_{2},\cdot\cdot\cdot$ be the elements of $\alpha$ of positive $\mu$ measure. We put

\begin{equation}\label{41}
H_{\mu}(\alpha)= \left\{
                                \begin{array}{rcl}
                                 -\Sigma_{l}\mu(A_{l})\log\mu(A_{l})&  & \text{if}\ \mu(X\backslash \bigcup_{l}A_{l})=0,\\
                                    +\infty   &  &\text{if}\ \mu(X\backslash \bigcup_{l}A_{l})>0.
                                \end{array}
                               \right.
\end{equation}

The above sum can be finite and infinite. $H_{\mu}(\alpha)$ is called the entropy of $\alpha$.
\end{definition}

\begin{definition}
Let $\alpha$ and $\beta$ are two measurable partitions of $X$, then almost every partition $\alpha_{B}$, defined as the restriction $\alpha|_{B}$ of $\alpha$ to $B$, $B\in X/\beta$, has a well-defined entropy $H_{\mu_{B}}(\alpha_{B})$. There is a non-negative measurable function on the factor-space $X/\beta$, called the conditional entropy of $\alpha$ with respect to $\beta$. Put

\begin{equation}\label{42}
H_{\mu}(\alpha|\beta)=\int_{X/\beta}H_{\mu_{B}}(\alpha_{B})d\mu_{\beta}.
\end{equation}

This integral can be finite or infinite. We call it the mean conditional entropy of  $\alpha$ with respect to $\beta$. Obviously, when $\beta$ is the trivial partition whose single element is $X$, $H_{\mu}(\alpha|\beta)$ coincides with the entropy $H_{\mu}(\alpha)$.
\end{definition}

\begin{prop}\label{22}
Let $G\curvearrowright^{T} (X, \mu)$ be a p.m.p. action on a Lebesgue space $(X,\mathcal{B},\mu)$, $\gamma$ be a measurable paitition, then
\begin{enumerate}
\item For any measurable partitions $\alpha,\beta$ of $X$, $H_{\mu}(\alpha\vee\beta|\gamma)\leq H_{\mu}(\alpha|\gamma)+H_{\mu}(\beta|\gamma)$;

\item For any measurable partitions $\alpha$ of $X$, $H_{\mu}(T_{g}\alpha|T_{g}\gamma)=H_{\mu}(\alpha|\gamma)$ for each $g\in G$;

\item For any measurable partitions $\alpha,\beta$ and $\gamma\leq\beta$ of $X$, implies $H_{\mu}(\alpha|\beta)\geq H_{\mu}(\alpha|\gamma)$ and $H_{\mu}(\gamma|\alpha)\leq H_{\mu}(\beta|\alpha)$;

\item Let $f:X\rightarrow X$ be a p.m.p. transformation, then $H_{\mu}(\alpha|\gamma)=H_{\mu}(f\alpha|f\gamma)$;

\item For any measurable partitions $\alpha,\beta,\beta$ of $X$, $H_{\mu}(\alpha\vee\beta|\beta)=H_{\mu}(\beta|\beta)+H_{\mu}(\alpha|\beta\vee\beta)$;

\item  For any measurable partitions $\eta_{n}$($n\in \mathbb{N}$), $\xi$ of $X$, if $\eta_{1}\leq\eta_{2}\leq\cdot\cdot\cdot$ and $\bigvee^{\infty}_{n=1}\eta_{n}=\eta$ and $H_{\mu}(\xi|\eta_{1})<+\infty$, then $H_{\mu}(\xi|\eta_{n})\rightarrow H_{\mu}(\xi|\eta)$.
\end{enumerate}
\end{prop}

%------------------------------------------------------------------------------
\section{Conditional entropy for amenable group actions}

For the conditional entropy of the system for amenable group actions, Yan \cite{KY} in 2015
extended the results of conditional entropy in \cite{DS,LW,Z}  to the infinite discrete countable amenable
group actions. Distinguishing from the method of Yan, in this paper, we introduce conditional entropy for amenable group actions by the method of factor space.

Let $G\curvearrowright^{T} (X, \mu)$ be a p.m.p. action on a Lebesgue space $(X,\mathcal{B},\mu)$ and $\mathcal{A}$ a sub-$\sigma$-algebra of $\mathcal{B}$ satisfying
\begin{equation}\label{1}
G\mathcal{A}\subseteq \mathcal{A},
\end{equation}
where $G \mathcal{A}=\{sA:s\in G,A\in\mathcal{A}\}$ and $sA=\{sx:x\in X\}$.
Then there exists a unique (mod 0) measurable partition $\mathcal{C}$ such that $\overline{\mathcal{B}}(\mathcal{C})$ is the $\sigma-$algebra $\overline{\mathcal{A}}$ consisting of classes of sets in $\mathcal{A}$, and from (\ref{1}) it follows that
\begin{equation}\label{2}
G\mathcal{C}\leq\mathcal{C},
\end{equation}

We denote
$$
\mathcal{Z}(\mathcal{C})=\{\alpha:\text{measurable partition}\ \alpha \text{ of } X \text{ satisfying } H_{\mu}(\alpha|\mathcal{C})<+\infty\}
$$

To prove Theorem \ref{3}, we now need Theorem \ref{51} and  Theorem \ref{52} which from \cite{KL1}.

\begin{theorem}\label{51}
Let $\varphi$ be a real-valued functions on the set of all nonempty finite subsets of $G$ satisfying

\begin{enumerate}
\item $\varphi(Fs)=\varphi(F)$ for all nonempty finite sets $F\subseteq G$ and $s\in G$, and;

\item $\varphi(F)\leq \frac{1}{k}\sum_{E\in \mathcal{K}}\varphi(E)$ for every nonempty finite set $F\subseteq G$ and $k-$ cover $\mathcal{K}$ of $F$ such that $\emptyset\neq \mathcal{K}$ and $E\subseteq F$ for all $E\notin \mathcal{K}$.
\end{enumerate}

Then $\varphi(F)/|F|$ converges to a limit as $F$ becomes more and more invariant and this limit is equal to
$$\inf\frac{\varphi(F)}{|F|},$$
where $F$ ranges over all nonempty finite subsets of $G$.
\end{theorem}

\begin{theorem}\label{52}
Let $\varphi$ be a $[0,\infty)$- valued function on the set of all finite subsets of $G$ such that for all finite sets $E,F\subseteq G$ one has

\begin{enumerate}
\item $\varphi(E)\leq\varphi(F)$ wherever $E\subseteq F$, and;

\item $\varphi(E\cup F)\leq \varphi(E)+\varphi(F)-\varphi(E\cap F).$
\end{enumerate}

Then
$$\varphi(F)\leq \frac{1}{k}\sum_{E\in \mathcal{K}}\varphi(E)$$
for every nonempty finite set $F\subseteq G$ and $k-$ cover $\mathcal{K}$ of $F$.

\end{theorem}

\begin{theorem}\label{3}
Let $G\curvearrowright^{T} (X, \mu)$ be a p.m.p. action on a Lebesgue space $(X,\mathcal{B},\mu)$, $\mathcal{A},\mathcal{C}$ be as given above, $F$ be a nonempty finite subset of $G$. Then for any $\alpha\in \mathcal{Z}(\mathcal{C})$
$$
\frac{1}{|F|}H_{\mu}(\alpha^{F}|\mathcal{C})
$$
converges to a limit as $F$ becomes more and more invariant and this limit is equal to
$$
\inf_{F}\frac{1}{|F|}H_{\mu}(\alpha^{F}|\mathcal{C})
$$
This limit is called the $\mathcal{A}$ conditional entropy of action $G\curvearrowright (X, \mu)$ with respect to $\alpha$, We define $h^{\mathcal{A}}_{\mu}(T,\alpha)$ to be the above limit.

\end{theorem}

\begin{proof}
Let
$$f(F)=\frac{1}{|F|}H_{\mu}(\alpha^{F}|\mathcal{C}).$$

By Proposition \ref{22}, for each finite partition $\alpha$ on $X$, the function $F\longrightarrow H_{\mu}(\alpha^{F}|\mathcal{C})$ defined on the collection of finite subsets of $G$ satisfies two conditions in Theorem \ref{51}  through Theorem \ref{52}, with the strong subadditivity in Theorem \ref{52} following from the observation that
\begin{equation*}
\begin{split}
H_{\mu}(\alpha^{E\cup F}|\mathcal{C})-H_{\mu}(\alpha^{E}|\mathcal{C})&=H_{\mu}(\alpha^{F\setminus E}|\alpha^{E}\vee\mathcal{C})\\
&\leq H_{\mu}(\alpha^{F\setminus E}|\alpha^{F\cap E}\vee\mathcal{C})\\
&=H_{\mu}(\alpha^{F}|\mathcal{C})-H_{\mu}(\alpha^{E\cap F}|\mathcal{C}),
\end{split}
\end{equation*}

The property of the mean conditional entropy in \cite{Roh} was used in the above proof process. In particular, we can express these quantities by taking the limit or infimum over any F{\o}lner sequence instead.
\end{proof}

\begin{definition}
Let $G\curvearrowright^{T} (X, \mu)$ be a p.m.p. action on a Lebesgue space $(X,\mathcal{B},\mu)$ and $\mathcal{A}$ be as given before. Then
$$
h^{\mathcal{A}}_{\mu}(T)=\sup_{\alpha\in\mathcal{Z}(\mathcal{C})}h^{\mathcal{A}}_{\mu}(T,\alpha)
$$
is called the $\mathcal{A}$-conditional entropy of the action $G\curvearrowright^{T} (X, \mu)$.
\end{definition}

\begin{theorem}\label{7}
Let $G\curvearrowright^{T}(X, \mu)$, $\mathcal{A}$ and $\mathcal{C}$ be as given before. Then for any $\alpha,\beta\in\mathcal{Z}(\mathcal{C})$, the following hold true
\begin{enumerate}
\item $h^{\mathcal{A}}_{\mu}(T,\alpha)\leq H_{\mu}(\alpha|\mathcal{C})$;

\item $h^{\mathcal{A}}_{\mu}(T,\alpha\vee\beta)\leq h^{\mathcal{A}}_{\mu}(T,\alpha)+h^{\mathcal{A}}_{\mu}(T,\beta)$;

\item $\alpha\leq\beta$, implies $h^{\mathcal{A}}_{\mu}(T,\alpha)\leq h^{\mathcal{A}}_{\mu}(T,\beta)$;

\item $h^{\mathcal{A}}_{\mu}(T,\alpha)\leq h^{\mathcal{A}}_{\mu}(T,\beta)+H_{\mu}(\alpha|\beta\vee\mathcal{C})$.

\end{enumerate}

\end{theorem}

\begin{proof}

Let $E,F$ be any nonempty finite subsets of $G$.

(1) From $\frac{1}{|F|}H_{\mu}(\alpha^{F}|\mathcal{C})\leq \frac{1}{|F|}\sum_{s\in F}H_{\mu}(s^{-1}\alpha|\mathcal{C})\leq\frac{1}{|F|}\sum_{s\in F}H_{\mu}(s^{-1}\alpha|s^{-1}\mathcal{C})=H_{\mu}(\alpha|\mathcal{C})$;

(2) Since $H_{\mu}((\alpha\vee\beta)^{F}|\mathcal{C})=H_{\mu}(\alpha^{F}\vee\beta^{F}|\mathcal{C})\leq H_{\mu}(\alpha^{F}|\mathcal{C})+H_{\mu}(\beta^{F}|\mathcal{C})$ then (2) holds obviously;

(3) $\alpha\leq\beta$, implies that $\alpha^{F}\leq\beta^{F}$ for all nonempty finite subset $F$ of $G$, then (3) follows from  theorem \ref{3};

(4) Since

$$
H_{\mu}(\alpha^{F}|\mathcal{C})\leq H_{\mu}(\alpha^{F}\vee\beta^{F}|\mathcal{C})=H_{\mu}(\beta^{F}|\mathcal{C})+H_{\mu}(\alpha^{F}|\beta^{F}\vee\mathcal{C})
$$
and
\begin{equation*}
\begin{split}
H_{\mu}(\alpha^{F}|\beta^{F}\vee\mathcal{C})&\leq\sum_{s\in F}H_{\mu}(s^{-1}\alpha|\beta^{F}\vee\mathcal{C})\\
&\leq\sum_{s\in F}H_{\mu}(s^{-1}\alpha|s^{-1}\beta\vee\mathcal{C})\\
&\leq\sum_{s\in F}H_{\mu}(s^{-1}\alpha|s^{-1}(\beta\vee\mathcal{C}))\\
&=|F|H_{\mu}(\alpha|\beta\vee\mathcal{C}),
\end{split}
\end{equation*}
this together with theorem \ref{3} yields (4).
\end{proof}

\begin{theorem}\label{4}
Let $G\curvearrowright^{T}(X, \mu)$, $\mathcal{A}$ be as given before. If $\alpha_{1}\leq\alpha_{2}\leq\cdot\cdot\cdot$ is an increase sequence of partitions in $\mathcal{Z}(\mathcal{C})$ such that $(\bigvee^{+\infty}_{n=1}\alpha_{n})\vee\mathcal{C}=\varepsilon$, where $\varepsilon$ denote a partition of $X$ into distinct points, Then
$$
h^{\mathcal{A}}_{\mu}(T,\alpha_{n})\longrightarrow h^{\mathcal{A}}_{\mu}(T), \text{as}\  n\longrightarrow +\infty.
$$
\end{theorem}

\begin{proof}
For any $\beta\in \mathcal{Z}(\mathcal{C})$ by (4) of theorem \ref{7}
$$
h^{\mathcal{A}}_{\mu}(T,\beta)\leq h^{\mathcal{A}}_{\mu}(T,\alpha_{n})+H_{\mu}(\beta|\alpha_{n}\vee\mathcal{C}).
$$

Since, according the property of mean conditional entropy in \cite{Roh}
$$
H_{\mu}(\beta|\alpha_{n}\vee\mathcal{C})\longrightarrow0, \text{as}\  n\longrightarrow+\infty,
$$
by (3) of theorem \ref{7} the sequence $h^{\mathcal{A}}_{\mu}(T,\alpha_{1}),h^{\mathcal{A}}_{\mu}(T,\alpha_{2}),\cdot\cdot\cdot$ is increasing, we have
$$
h^{\mathcal{A}}_{\mu}(T,\beta)\leq\lim_{ n\longrightarrow+\infty}h^{\mathcal{A}}_{\mu}(T,\alpha_{n}).
$$
Since $\beta$ is arbitrary in $\mathcal{Z}(\mathcal{C})$ we see that
$$
h^{\mathcal{A}}_{\mu}(T,\alpha_{n})\longrightarrow h^{\mathcal{A}}_{\mu}(T), \text{as}\  n\longrightarrow+\infty,
$$
Completing the proof.
\end{proof}

A measurable partitions $\alpha$ of $X$ is said to be fixed under an group $G$ action if every element $A$ of $\alpha$ satisfies $GA=A$.
\begin{theorem}\label{5}
Let $G\curvearrowright^{T}(X, \mu)$, $\mathcal{A}$, $\mathcal{C}$  be as given before. and assume that $G\mathcal{A}=\mathcal{A}$. If $\alpha\in \mathcal{Z}(\mathcal{C})$ and $\beta$ is a measurable partition fixed under the action $G\curvearrowright^{T}(X, \mu)$. Then
$$
h^{\mathcal{A}}_{\mu}(T,\alpha)=\int_{X/\beta}h^{\mathcal{A}_{B}}_{\mu_{B}}(T_{B},\alpha_{B})d\mu_{\beta}.
$$
where $\mathcal{A}_{B}=\mathcal{A}|_{B}$.
\end{theorem}

\begin{proof}
Let $\{F_{n}\}$ be a F{\o}lner sequence of $G$. For any $x\in X$, define $m(x,\alpha|\mathcal{C})=\mu_{\mathcal{C}(x)}(\alpha(x)\cap\mathcal{C}(x) )$, then it is measurable function on $X$ and \cite{42} can be written in the form:
$$H_{\mu}\alpha|\mathcal{C}=-\int_{X}\log m(x,\alpha|\mathcal{C})d\mu.$$

Then
\begin{equation*}
\begin{split}
h^{\mathcal{A}_{B}}_{\mu_{B}}(T_{B},\alpha_{B})&=\lim_{n\longrightarrow+\infty}\frac{1}{|F_{n}|}H_{\mu_{B}}(\alpha^{F_{n}}_{B}|\mathcal{C}_{B})\\
&=\lim_{n\longrightarrow+\infty}\frac{1}{|F_{n}|}\left[-\int_{B}\log\mu_{\mathcal{C}_{B}(x)}\left(\alpha^{F_{n}}_{B}(x)\cap\mathcal{C}_{B}(x)\right)d\mu_{B}\right],
\end{split}
\end{equation*}
and so
\begin{equation*}
\begin{split}
\int_{X/\beta}h^{\mathcal{A}_{B}}_{\mu_{B}}(T_{B},\alpha_{B})d\mu_{\beta}&=\int_{X/\beta}\lim_{n\longrightarrow+\infty}\frac{1}{|F_{n}|}H_{\mu_{B}}(\alpha^{F_{n}}_{B}|\mathcal{C}_{B})d\mu_{\beta}\\
&=-\lim_{n\longrightarrow+\infty}\frac{1}{|F_{n}|}\int_{X}\log\mu_{\mathcal{C}(x)}\left(\alpha^{F_{n}}(x)\cap\mathcal{C}(x)\right)d\mu\\
&=\lim_{n\rightarrow\infty}\frac{H_{\mu}(\alpha^{F_{n}}|\mathcal{C})}{|F_{n}|}\\
&=h^{\mathcal{A}}_{\mu}(T,\alpha).
\end{split}
\end{equation*}

\end{proof}

\begin{theorem}\label{31}
Let $G\curvearrowright^{T}(X, \mu)$, $\mathcal{A}$, be as given before and assume that $G\mathcal{A}=\mathcal{A}$. If $\beta$ is a measurable partition fixed under the action $G\curvearrowright^{T}(X, \mu)$. Then
\begin{equation}\label{8}
h^{\mathcal{A}}_{\mu}(T)=\int_{X/\beta}h^{\mathcal{A}_{B}}_{\mu_{B}}(T_{B})d\mu_{\beta}.
\end{equation}
where $\mathcal{A}_{B}=\mathcal{A}|_{B}$.
\end{theorem}

\begin{proof}
Let $\beta_{1}\leq\beta_{2}\leq\cdot\cdot\cdot$ is an increase sequence of finite measurable partitions of $X$ such that $\beta_{n}\longrightarrow \varepsilon$. Then $\beta_{n}\in \mathcal{Z}(\mathcal{C})$. According to the Theorem \ref{5}, for any $n\in \mathbb{N}$ we have
\begin{equation}\label{9}
h^{\mathcal{A}}_{\mu}(T,\beta_{n})=\int_{X/\beta}h^{\mathcal{A}_{B}}_{\mu_{B}}(T_{B},(\beta_{n})_{B})d\mu_{\beta}.
\end{equation}
From $\beta_{n}\longrightarrow \varepsilon$, then $(\bigvee_{n=1}^{\infty}\beta_{n})\vee \mathcal{C}=\varepsilon$ and $(\bigvee_{n=1}^{\infty}(\beta_{n})_{B})\vee \mathcal{C}=\varepsilon_{B}$. By theorem \ref{4} for $\mu_{\beta}-$a.e. $B\in X/\beta$
\begin{equation}\label{10}
h^{\mathcal{A}}_{\mu}(T,\beta_{n})\longrightarrow h^{\mathcal{A}}_{\mu}(T)
\end{equation}
and
\begin{equation}\label{11}
h^{\mathcal{A}_{B}}_{\mu_{B}}(T_{B},(\beta_{n})_{B})\longrightarrow h^{\mathcal{A}_{B}}_{\mu_{B}}(T_{B}),
\end{equation}
as $n\longrightarrow+\infty$. Then (\ref{8}) follows from (\ref{9}), (\ref{10}) and (\ref{11}) the monotone convergence theorem.
\end{proof}

Let $T$ be a measure-preserving transformation on the Lebesgue space $(X,\mathcal{B},\mu)$, refer to \cite{LQ} for the corresponding results. $T$ is called ergodic if every measurable set $A$ satisfying $T^{-1}A=A$ has either measure 1 or measure 0. If $T$ is not ergodic, then it can be decomposed into ergodic components (see \cite{Roh1}). For amenable group actions see \cite{DZ}, the case $\mu$ is ergodic in Theorem \ref{31} is well known see \cite{D} and then it is not hard to obtain Theorem \ref{31} in the general case by applying ergodic decomposition.

%Let $\{F_{n}\}$ be a F{\o}lner sequence of $G$ with the property that $e_{G}\in F_{1}\subset F_{2}\subset\cdot\cdot\cdot$ and hence $|F_{n}|\geq n$ for each $n\in \mathbb{N}$. Write $\beta_{i}=\beta^{F_{i}}$ for $i\in \mathbb{N}$.

%\bigskip \noindent{\bf Acknowledgement}.

%------------------------------------------------------------------------------

%\bibliographystyle{amsplain}

\end{document}